\documentclass[a4paper,11pt]{article}
\usepackage{a4wide}
\usepackage{amsfonts}
\usepackage{graphics}
\usepackage{amsmath, amsthm}  
\usepackage{amssymb,bbm} 
\usepackage{enumerate}
\usepackage{amsthm}
\usepackage{hyperref}
\usepackage{cite}
\usepackage[all]{xy}

\usepackage[margin=1in]{geometry}

\hypersetup{
    bookmarks=true,         
    unicode=false,          
    pdftoolbar=true,        
    pdfmenubar=true,        
    pdffitwindow=false,     
    pdfstartview={FitH},    
    pdftitle={Linear groups, nilpotent Lie algebras, and identities},    
    pdfauthor={Wolfgang Alexander Moens},     
    pdfsubject={Survey},   
    pdfkeywords={Groups, Lie algebras, Gradings}, 
    pdfnewwindow=true,      
    colorlinks=true,       
    linkcolor=black,          
    citecolor=black,        
    filecolor=magenta,      
    urlcolor=gray           
}

\usepackage{tikz}
\usepackage{epstopdf}

\newcommand{\N}{\mathbb{N}}
\newcommand{\F}{\mathbb{F}}

\newcommand{\Lf}{\mathfrak{f}}
\newcommand{\Lg}{\mathfrak{g}}\newcommand{\Lh}{\mathfrak{h}}
\newcommand{\Li}{\mathfrak{i}}
\newcommand{\Ll}{\mathfrak{l}}\newcommand{\Lm}{\mathfrak{m}}\newcommand{\Ln}{\mathfrak{n}}
\newcommand{\Lp}{\mathfrak{p}}
\newcommand{\Lr}{\mathfrak{r}}

\newcommand{\CB}{\mathcal{B}}

\newcommand{\CQ}{\mathcal {Q}}
\newcommand{\CS}{\mathcal {S}}\newcommand{\CU}{\mathcal {U}}

\newcommand{\rad}{\operatorname{rad}}

\newcommand{\map}[3]{ #1 : #2 \longrightarrow #3 }

\newtheorem{theorem}{Theorem}[subsection]

\newtheorem{proposition}[theorem]{Proposition}

\newtheorem{corollary}[theorem]{Corollary}
\newtheorem*{corollary*}{Corollary}

\newtheorem{lemma}[theorem]{Lemma}

\newtheorem*{theorem*}{Theorem}

\newtheorem*{proposition*}{Proposition}
\newtheorem*{lemma*}{Lemma}
\newtheorem*{problem*}{Problem}

\newtheorem*{observation*}{Observation}

\theoremstyle{definition} \newtheorem{remark}[theorem]{Remark}
\theoremstyle{definition} \newtheorem{example}[theorem]{Example}
\theoremstyle{definition}\newtheorem{definition}[theorem]{Definition}
\theoremstyle{definition}

\title{Representing Lie algebras\\using approximations with nilpotent ideals}

\author{Wolfgang Alexander Moens\thanks{This research was supported by the Austrian Science Foundation FWF (Grant J3371-N25).}}
\date{2015}

\begin{document}

\maketitle

\abstract{We prove a refinement of Ado's theorem for Lie algebras over an algebraically-closed field of characteristic zero. We first define what it means for a Lie algebra $\Lg$ to be approximated with a nilpotent ideal, and we then use such an approximation to construct a faithful representation of $\Lg$. The better the approximation, the smaller the degree of the representation will be. We obtain, in particular, explicit and combinatorial upper bounds for the minimal degree of a faithful $\Lg$-representation. The proofs use the universal enveloping algebra of Poincar\'e-Birkhoff-Witt and the almost-algebraic hulls of Auslander and Brezin.}

\tableofcontents


\linespread{1.2}

\section{Introduction}

The classical theorem of Ado and Iwasawa states that every finite-dimensional Lie algebra $\Lg$ over a field $\F$ admits a finite-dimensional, faithful representation. That is: there exists a natural number $n$ and a homomorphism $\map{\varphi}{\Lg}{\Lg\Ll_n(\F)}$ of Lie algebras with $\operatorname{ker}(\varphi) = \{ 0 \}$. But it is in general quite difficult to determine whether a given Lie algebra $\Lg$ admits a faithful representation of a given degree $n$. \newline

The problem is easily seen to be equivalent to the computation of the minimal degree $\mu(\Lg)$ of a faithful representation of $\Lg$. The study of this invariant was crucial in finding (filiform) counter-examples to a conjecture by Milnor on the existence of affine structures on manifolds, \cite{Milnor}, \cite{Benoist} and \cite{BurdeGrunewald}. Lower and upper bounds for $\mu(\Lg)$ in terms of other natural invariants of $\Lg$ were given for various families of Lie algebras by Benoist, Birkhoff, Burde, de Graaf, Eick, Grunewald and the author, \cite{Birkhoff}, \cite{BurdeAdo}, \cite{BEG_Algorithm}, \cite{BurdeMoensReductive}, \cite{BurdeMoensQuotient}, and \cite{deGraaf}. This paper aims to refine the current upper bounds. \newline

We recall two major techniques that have been used historically to construct faithful representations. A nilpotent Lie algebra $\Ln$ is known to act faithfully on its universal enveloping algebra $\mathcal{U}(\Ln)$, and Birkhoff observed that the subspace $S$ of all sufficiently large elements (made precise in the next sections), yields a faithful quotient module $\mathcal{U}(\Ln) / S$ of dimension $\frac{d^{c+1}-1}{d-1}$,
 where $d$ is the dimension of $\Ln$ and $c$ is its class. Mal'cev later observed that Ado's theorem is a natural consequence of the existence of almost-algebraic hulls (also called splittings) and Birkhoff's construction. A constructive approach to Mal'cev's theorem by Neretin made use of elementary expansions and it allowed Burde and the author to find explicit upper bounds for $\mu(\Lg)$ for all finite-dimensional complex Lie algebras $\Lg$: $ \mu(\Lg) = O(2^d). $ \newline

Several examples in the literature then suggested that Lie algebras which have an abelian ideal of small codimension in the solvable radical also have a small, faithful representation (see for example \cite{BurdeMoensReductive} and propositions $2.12$, $2.15$, $4.5$ and remark $4.8$ of \cite{BurdeMoensQuotient}). This was made precise and proven by Burde and the author to obtain bounds of the form $\mu(\Ln) = O(d^{\gamma + 1})$, where $\Ln$ is a $d$-dimensional nilpotent Lie algebra and $\gamma$ is the minimal codimension of an abelian ideal. \newline

We will prove that these results can be generalised to nilpotent ideals of arbitrary class:

\begin{theorem} \label{TheoremMain} Consider a Lie algebra $\Lg$ over an algebraically-closed field of characteristic zero. Let $d$ be its dimension, let $r$ be the dimension of the solvable radical and let $n$ be the dimension of the nilradical. Suppose $\Lg$ has a nilpotent ideal of class $\varepsilon_1$ and codimension $\varepsilon_2$ in $\rad(\Lg)$. Then $$\mu(\Lg) \leq d - n + \binom{r + \varepsilon_1}{\varepsilon_1} \cdot \binom{r + \varepsilon_2}{\varepsilon_2}.$$
\end{theorem}




We note that de Graaf's theorem, \cite{deGraaf}, corresponds with the special case where the Lie algebra is itself nilpotent and the ideal is chosen to be the whole Lie algebra: $d := n$, $\varepsilon_1 := c(\Lg)$ and $\varepsilon_2 := 0$. We define the nil-defect $\varepsilon = \varepsilon(\Lg)$ of $\Lg$ to be the minimal value $\varepsilon_1 + \varepsilon_2$ as $\Lh$ runs over all nilpotent ideals of $\Lg$.

\begin{corollary} \label{CorPolynomial} Consider a Lie algebra $\Lg$ over an algebraically-closed field of characteristic zero. Let $d$ be its dimension and let $\varepsilon$ be its nil-defect. Then $\Lg$ has a faithful representation of degree $$ P_\varepsilon(d) := d + \frac{(d+\varepsilon) \cdots (d+1)}{\lfloor \frac{\varepsilon}{2} \rfloor ! \cdot \lceil \frac{\varepsilon}{2} \rceil!}.$$
\end{corollary}

We can also apply the construction to graded Lie algebras; there exists a function $\map{E}{\N \times \N}{\N}$ such that the following holds.

\begin{corollary} \label{CorGrading} Consider a Lie algebra $\Lg = \oplus_{a \in A} \Lg_a$ graded by an abelian, finitely-generated, torsion-free group $(A,+)$. Let $\sigma := |\{ a \in A | \Lg_a \neq \{ 0 \} \}|$ be the cardinality of the support and let $\delta := \dim(\Lg_0)$ be the dimension of the homogeneous component corresponding with the neutral element. Then 
$\Lg$ admits a faithful representation of degree $P_{E(\sigma,\delta)}(d).$
\end{corollary}

\emph{Convention}: we will only consider finite-dimensional Lie algebras over an algebraically-closed field of characteristic zero.

\section{Preliminaries}

In this section we introduce some of the concepts used in our construction of faithful representations. We first define the nil-defect $\varepsilon(\Lg)$ of a Lie algebra $\Lg$. We then show that $\Lg$ can (almost) be embedded into an almost-algebraic Lie algebra $\widehat{\Lg}$ for which the nil-defect $\varepsilon(\widehat{\Lg})$ is at most $\varepsilon(\Lg)$.

\subsection{The nil-defect of a Lie algebra}

The following definition is justified by the existence of a (unique) solvable radical, that is: the theorem of Levi-Malcev.

\begin{definition}[Nil-defect] Let $\Lr$ be a solvable Lie algebra and $\Ln$ be a nilpotent ideal of $\Lr$. The nil-defect $\varepsilon(\Lr,\Ln)$ of $\Ln$ in $\Lr$ is $\dim(\Lr / \Ln) + c(\Ln)$. The nil-defect $\varepsilon(\Lr)$ of $\Lr$ is $$ \min_{\Ln} \{ \varepsilon(\Lr,\Ln) \},$$ where $\Ln$ runs over all nilpotent ideals of $\Lr$. The nil-defect of an arbitrary Lie algebra is the nil-defect of its solvable radical.
\end{definition}

Let us consider a few special cases.

\begin{example} The nil-defect of a semisimple Lie algebra is $0$. The nil-defect of a nilpotent Lie algebra $\Ln$ is bounded by the nilpotency class: $\varepsilon(\Ln) \leq \varepsilon(\Ln,\Ln) = c(\Ln)$. In particular: the family of all Lie algebras of nil-defect at most $\varepsilon$ contains the family of all nilpotent Lie algebras of class at most $\varepsilon$. \end{example}


We note that Lie algebras of a given nil-defect can have arbitrarily high nilpotency class.

\begin{example} A standard filiform Lie algebra has a nil-defect of $2$, while its class can be chosen arbitrarily high. More generally: for filiform Lie algebras $\Lf$, we have $\varepsilon(\Lf) \leq 2 \sqrt{\dim(\Lf)} + 1$. This is a direct consequence of the fact that any ideal of $\Lf$ with codimension $a > 1$ is nilpotent of class at most $\lceil (c(\Lf)+1)/a \rceil - 1$, cf. \cite{Vergne}. \end{example}

The following result, not strictly necessary for the rest of the paper, is due to B. Kostant and allows us to approximate solvable Lie algebras with nilpotent \emph{subalgebras}, rather than ideals (personal communication with G. Glauberman and N. Wallach; see for example \cite{Kostant} and \cite{AlperinGlauberman}. See also \cite{BurdeCeballos}.) 

\begin{theorem} Consider a finite-dimensional complex, \emph{solvable} Lie algebra $\Lr$ and a nilpotent subalgebra $\Ln$. Then $\Lr$ has an ideal $\Lh$ of class at most $c(\Ln)$ and dimension equal to $\dim(\Ln)$. In particular: $$ \varepsilon(\Lr) = \min_{\Lm} \{ c(\Lm) + \operatorname{codim}_{\Lr}(\Lm) \},$$ where $\Lm$ runs over the nilpotent subalgebras of $\Lr$. \end{theorem}


\subsection{Almost-algebraic Lie algebras}


\begin{definition}[Almost-algebraic] A Lie algebra $\Lg$ is almost-algebraic if it admits a decomposition of the form $\Lg = \Lp \ltimes \Lm$, where $\Lm$ is the nilradical of $\Lg$ and $\Lp$ is a subalgebra of $\Lg$ that acts fully reducibly on $\Lg$ (by the adjoint representation). \end{definition}

A theorem by Mal'cev, later generalized by Auslander and Brezin, states that every finite-dimensional Lie algebra over an algebraically closed field of characteristic zero can be embedded into an almost-algebraic Lie algebra, and that there is a minimal such algebra: the \emph{almost-algebraic hull}, \cite{Malcev}, \cite{AuslanderBrezin}, \cite{OnishchikVinberg}. Neretin later gave an explicit construction, as a succession of finitely many elementary expansions, of such an embedding, \cite{Neretin}. This construction was used by Burde and the author to obtain explicit upper bounds for $\mu(\Lg)$, \cite{BurdeMoensQuotient}. Auslander and Brezin observed that ideals are compatible with elementary expansions:

\begin{lemma} Let $\map{\iota_1}{\Lg_1}{\Lg_2}$ be an elementary expansion of $\Lg_1$. Then every ideal $\Li$ of $\Lg_1$ maps onto an ideal $\iota_1(\Li)$ of $\Lg_2$. \end{lemma}

In particular: if $(\map{\iota_j}{\Lg_j}{\Lg_{j+1}})_{1 \leq j \leq u}$ is a finite sequence of elementary expansions, and $i =: \iota_u \circ \iota_{u-1} \circ \cdots \circ \iota_1$, then every ideal $\Li$ of $\Lg_1$ maps onto an ideal $\iota(\Li)$ of $\Lg_{u+1}$. We may thus combine the lemma with proposition $4.1$ of \cite{BurdeMoensQuotient} to obtain the following theorem.


\begin{proposition}[Good embedding] \label{PropEmbedding} Let $\Lg$ be a finite-dimensional Lie algebra over the complex numbers. Let $\Lr$ be its solvable radical and let $\Ln$ be its nilradical. Then there exists an embedding $\map{\iota}{\Lg}{\widehat{\Lg}}$ of $\Lg$ into the Lie algebra $\widehat{\Lg}$ such that:
\begin{enumerate}
\item (Decomposition): $\widehat{\Lg}$ decomposes as $\Lp \ltimes \Lm$, where $\Lm$ is nilpotent and $\Lp$ acts fully reducibly on $\widehat{\Lg}$,
\item (Controll of dimensions): $\dim(\Lm) = \dim(\Lr)$ and $\dim(\Lp) = \dim(\Lg / \Ln)$,
\item (Preservation of ideals): if $\Lh$ is a nilpotent ideal of $\Lg$, then $\iota(\Lh)$ is a nilpotent ideal of $\Lp \ltimes \Lm$ contained in $\Lm$.
\end{enumerate}
\end{proposition}

\begin{proof} Points $(1)$ and $(2)$ can be obtained by expanding $\dim(\Lr / \Ln)$ times with respect to the nilradical, cf. $4.1$ of \cite{BurdeMoensQuotient}. Point $(3)$ follows from the lemma. \end{proof}

\begin{remark} \label{Remark} In the above decomposition $\Lp \ltimes \Lm$, the nilpotent ideal $\Lm$ need not be the whole nilradical. However, if we let $\Lp_0$ be the kernel of the action of $\Lp$ on $\Lm$, then $\Lp \ltimes \Lm \cong \Lp_0 \oplus ( \Lp/\Lp_0 \ltimes \Lm )$ and $\Lm$ will be the nilradical of the almost-algebraic algebra $(\Lp/\Lp_0 \ltimes \Lm)$. \end{remark}

In order to construct faithful representations of $\Lg$ it therefore suffices to construct (sufficiently small) faithful representations of $\widehat{\Lg}$. 

\section{Quotients of the universal enveloping algebra}

In this section we will construct faithful representations of almost-algebraic Lie algebras $\Lp \ltimes \Lm$. In order to do this we first introduce weight functions $\map{\omega}{\CU(\Lm)}{\N \cup \{ \infty \}}$ on the universal enveloping algebra $\CU(\Lm)$ of the nilpotent Lie algebra $\Lm$. We shall then see that the elements of $\CU(\Lm)$ that are sufficiently large with respect to such a weight function form a $(\Lp \ltimes \Lm)$-submodule $\CS'$ of $\CU(\Lm)$. A good choice of weight functions will allow us to construct a quotient $\CU(\Lm) / \CS$ that is faithful and finite-dimensional. 

\subsection{Filtrations}

In the following sections we will be working with pairs of filtrations of a given nilpotent Lie algebra. It will be convenient to have a basis that is compatible with those filtrations.

\begin{definition} A \emph{filtration} of a Lie algebra $\Lg$ is a flag $(\Lg(t))_{t \in \N}$ of subspaces of $\Lg$ of the form, $\Lg = \Lg(0) \supseteq \Lg(1) \supseteq \Lg(2) \supseteq \cdots $ such that for all $\Lg(i),\Lg(j)$ we have $[\Lg(i),\Lg(j)] \subseteq \Lg(i+j)$. It is a \emph{positive} filtration if $\Lg(0) = \Lg(1)$. \end{definition}

Note that each element of a filtration is an ideal of the Lie algebra. We will consider the following example in the next paragraphs.

\begin{example} Let $\Lh$ be an ideal of a Lie algebra $\Lg$. Then $\Lg(0) := \Lg$, $\Lh(1) := \Lh$ and $\Lg(i) := \Lh_i := [\Lh,\Lh_{i-1}] $ for $i \geq 2$ defines a filtration of $\Lg$. 
Let us call this the $(\Lg,\Lh)$-filtration. The $(\Lg,\Lh)$-filtration is clearly positive if $\Lg = \Lh$ is nilpotent. \end{example}

\begin{definition} Let $V$ be a vector space and consider a flag $(V_j)_{j \in \N}$ of $V$. We say that a basis $\CB$ of $V$ is \emph{weakly adapted} to the flag, iff for each $V_j$ there exists a subset $\CB_j$ of $\CB$ that is a basis of $V_j$. \end{definition}

Some elementary observations in linear algebra lead to the following.

\begin{lemma} Consider a Lie algebra $\Lg$ and a pair of $\Lg$-filtrations. Then $\Lg$ has a basis that is weakly adapted to both filtrations. \end{lemma}

The corresponding statement for triples of filtrations fails trivially.

\subsection{Notation}

Let us fix some notation. We let $\Lg$ be a fixed almost-algebraic Lie algebra with corresponding decomposition $\Lp \ltimes \Lm$ and a nilpotent ideal $\Lh$ of $\Lg$, necessarily contained in $\Lm$. Let us also fix the following pair of filtrations of $\Lm$: the $(\Lm,\Lm)$-filtration and the $(\Lm,\Lh)$-filtration of $\Lm$. The lemma above then allows us to choose a basis for $\Lm$ that is weakly adapted to both filtrations. Let $\{x_1,\ldots,x_d\}$ be such a basis. \newline

The Poincar\'e-Birkhoff-Witt theorem states that the standard (non-commutative, ordered) monomials in the $x_i$ form a basis for the universal enveloping algebra $\CU(\Lm)$ of $\Lm$. We also recall that $\Lp \ltimes \Lm$ acts naturally on $\CU(\Lm)$: $\Lp$ acts by derivations and $\Lm$ acts by left multiplication. To be precise: $$ \delta \ast (x_{i_1} \cdots x_{i_t}) := \sum_{1 \leq j \leq t} x_{i_1} \cdots x_{i_{j-1}} \cdot [\delta , x_{i_j}] \cdot x_{i_{j+1}} \cdots x_{i_t} $$
and $ x \ast (x_{i_1} \cdots x_{i_t}) := x \cdot x_{i_1} \cdots x_{i_t} $ for all $x \in \Lm$, $\delta \in \Lp$, and monomials $x_{i_1} \cdots x_{i_t}$. The $\Lp \ltimes \Lm$-module $\CU(\Lm)$ is faithful but infinite-dimensional.

\subsection{From filtrations to weights and submodules}

Let us first define weight functions on $\CU(\Lm)$ and then show how they can be constructed from positive filtrations on $\Lm$. Set $\overline{\N} := \N \cup \{ + \infty \}$.

\begin{definition} A map $\map{\omega}{\CU(\Lm)}{\overline{\N}}$ is a \emph{weight} on $\CU(\Lm)$ iff it satisfies the following conditions:
\begin{enumerate}
\item $\omega(X) = + \infty \Leftrightarrow X = 0$
\item $\omega(X + Y) \geq \min\{ \omega(X) , \omega(Y) \}$
\end{enumerate}
for all $X$ and $Y$ in $\CU(\Lm)$. The weight is \emph{compatible with the action} of $\Lp \ltimes \Lm$ on $\CU(\Lm)$ iff in addition the following conditions hold:
\begin{enumerate}
\item[3.] $\omega(x \ast X) \geq \omega(x) + \omega(X)$
\item[4.] $\omega(\delta \ast X) \geq \omega(X)$
\end{enumerate}
for all $X,Y$ in $\CU(\Lm)$, $x$ in $\Lm$ and $\delta$ in $\Lp$.
\end{definition}


We may then define subsets all elements of which are sufficiently large.

\begin{definition} Let $\map{\omega}{\CU(\Lm)}{\overline{\N}}$ be a map and let $k$ be a natural number. Then we define the set $$\CU^k(\Lm,\omega) := \{ X \in \CU(\Lm) | \omega(X) \geq k \}.$$ \end{definition}

Let us consider a filtration of $\Lg$ and show how it can be used to define a weight on $\CU(\Lm)$. Let $ \Lm(0) \supseteq \Lm(1) \supseteq \Lm(2) \supseteq \cdots $ be the filtration. For $x \in \Lm$ we define $$\omega(x) := \sup \{ t \in \overline{\N} | x \in \Lm(t) \}.$$ We obtain a map $\map{\omega}{\Lm}{\overline{\N}}$. 
We may now extend $\omega$ to all of $\CU(\Lm)$ as follows. For a standard monomial $X^{\alpha} := x_1^{\alpha_1} \cdots x_{d}^{\alpha_d}$, with $\alpha = (\alpha_1,\ldots,\alpha_d) \in \N^d$ we define $$ \omega(X^{\alpha}) := \sum_{1 \leq j \leq d} \alpha_j \cdot \omega(x_j),$$ and for a non-redundant linear combination $\sum_j \varphi_j X^{\alpha_j} $ of standard monomials $X^{\alpha_j}$ with linear coefficients $\varphi_j$ we set $$\omega(\sum_j \varphi_j X^{\alpha_j}) := \min_j \{ \omega(X^{\alpha_j}) \}.$$ In particular, we may consider special cases:

\begin{example} Suppose $\Ln$ is a nilpotent ideal of $\Lm$, for example $\Lm$ itself or $\Lh$. 
We then let $\omega_{(\Lm,\Ln)}$ be the weight on $\CU(\Lm)$ obtained from the $(\Lm,\Ln)$-filtration. We let $\lambda$ be the weight on $\CU(\Lm)$ obtained from the trivial, positive filtration $\Lm := \Lm(0) := \Lm(1)$ and $\Lm(i) := \{ 0 \}$ for $i \geq 2$. \end{example}

This $\lambda$ can be considered a length function (on the standard monomials).

\begin{lemma} If $\Ln$ is a nilpotent ideal of $\Lp \ltimes \Lm$ contained in $\Lm$, then $\omega_{(\Lm,\Ln)}$ is a weight on $\CU(\Lm)$ that is compatible with the action of $\Lp \ltimes \Lm$. \end{lemma}

\begin{proof} Property $(1)$ holds since $\Ln$ is nilpotent (and only if the ideal is nilpotent). Since $\CU^k(\Lm,\omega_{(\Lm,\Ln)})$ is spanned by the standard monomials of degree at least $k$, we also have the second property. If $x_{i}$ and $x_j$ are basis vectors with $j < i$, then we have $\omega(x_i \cdot x_j) = \omega(x_j \cdot x_i + [x_i,x_j]) \geq \min(\omega(x_j \cdot x_i),\omega([x_i,x_j])) \geq \omega(x_i) + \omega(x_j)$, by using the definition of $\omega$ on standard monomials and the property of the grading on $\Lm$. By using induction on the length of a standard monomials (and $(2)$), we then obtain property $(3)$. Property $(4)$ follows from $(3)$ and the fact that $\delta$ stabilises the flag (since $\Ln$ is an ideal of $\Lp \ltimes \Lm$). \end{proof}

We note that if a weight $\omega$ is compatible with the action of $\Lp \ltimes \Lm$, then each such subspace $\CU^k(\Lm,\omega)$ is a $(\Lp \ltimes \Lm)$-submodule of $\CU(\Lm)$. In particular: we conclude that for each $k_1,k_2$ in $\N$, $$\CS(\Lm,\omega_{(\Lm,\Lm)},k_1,\omega_{(\Lm,\Lh)},k_2) := \CU^{k_1}(\Lm,\omega_{(\Lm,\Lm)}) + \CU^{k_2}(\Lm,\omega_{(\Lm,\Lh)})$$ is a $(\Lp \ltimes \Lm)$-stable subspace of $\CU(\Lm)$ and we may consider the quotient module.

\subsection{Properties of the quotient module}

We now need to determine two things: the dimension of the quotient and for which choices the quotient will be faithful.

\begin{proposition}[Faithful quotient] \label{PropFaithful} Suppose $\Lp$ acts faithfully on $\Lm$. If $k_1 > c(\Lm)$ and $k_2 > c(\Lh)$, then the quotient of $\CU(\Lm) $ by $ \CS(\Lm,\omega_{(\Lm,\Lm)},k_1,\omega_{(\Lm,\Lh)},k_2)$ is a faithful $(\Lp \ltimes \Lm)$-module. \end{proposition}

\begin{proof} Note that the submodule $\CS$ is contained in the subspace $\Lambda_{\geq 2} := \CU^2(\Lm,\lambda)$ of $\CU(\Lm)$ that is spanned by all standard monomials of length at least two. Now suppose that $(\delta,x) \in \Lp \ltimes \Lm$ maps $\CU(\Lm)$ into $\CS$. Then $x = (\delta,x) \ast 1 = \delta(1) + x \ast 1 = 0 + x = x \in \CS \subseteq \Lambda_{\geq 2}$. Note that $\Lm$ is the vector space spanned by all the standard monomials of length one. 
We conclude that $x \in \Lambda_{\geq 2} \cap \Lm = \{ 0 \}$. Similarly, $\Lp$ maps $\Lm$ into $\Lambda_{\geq 2} \cap \Lm = \{0\}$. Since $\Lp$ is assumed to act faithfully on $\Lm$, we have $(\delta,x) = (0,0)$. \end{proof}

We recall the following well-known result about Sylvester denumerants.

\begin{lemma} Consider a finite multiset $M = \{ m_1,\ldots,m_p\}$ of positive integers and $t \in \N$. Then the number $\Delta(t;M)$ of $M$-partitions of $t$ is bounded from above by $\binom{p + t - 1}{t -1}.$ \end{lemma}

In particular: the number of $M$-partitions of $0 \leq t \leq T$ is at most $\binom{p + T}{T}$. The proposition now suggests the choice $k_1 := c(\Lm) + 1$ and $k_2 := c(\Lh) + 1$. We then get:

\begin{proposition}[Upper bound] \label{PropBound} The $(\Lp\ltimes\Lm)$-module $$\CQ := \CU(\Lm) / \CS(\Lm,\omega_{(\Lm,\Lm)},c(\Lm)+1,\omega_{(\Lm,\Lh)},c(\Lh)+1)$$ has dimension at most $$\binom{\dim(\Lm) + \dim(\Lm / \Lh)}{\dim(\Lm / \Lh)} \cdot \binom{\dim(\Lm) + c(\Lh)}{c(\Lh)}.$$
\end{proposition}

\begin{proof} Since $\CS$ is spanned by all standard monomials $X$ satisfying $\omega_{(\Lm,\Lm)}(X) \geq c(\Lm) + 1$ or $\omega_{(\Lm,\Lh)}(X) \geq c(\Lh) + 1$, the dimension of $\CQ$ is bounded from above by the number of standard monomials $Y$ satisfying $\omega_{(\Lm,\Lm)}(Y) \leq c(\Lm)$ and $\omega_{(\Lm,\Lh)}(Y) \leq c(\Lh)$. The lemma above then gives the crude upper bound $$\dim(\CQ) \leq \binom{\dim(\Lm / \Lh) + c(\Lm)}{c(\Lm)} \cdot \binom{\dim(\Lh) + c(\Lh)}{c(\Lh)}$$ and the identity $\binom{a+b}{b} = \binom{a+b}{a}$ finishes the proof. \end{proof}


We note that if a weight $\omega$ is given explicitly, it makes sense to compute the corresponding Sylvester denumerant directly.

\begin{example} If $\Lf$ is filiform, then we can find a decomposition $\Lp \ltimes \Lm$ of $\Lf$ (cf. \cite{Vergne} and \cite{BurdeMoensQuotient}) such that the number of standard monomials $X$ of $\CU(\Lm)$ satisfying $\omega_{(\Lm,\Lm)}(X) = t$ is given by the usual partition function $$p(t) \sim \frac{e^{\pi \sqrt{\frac{2 t}{3}}}}{4 t \sqrt{3}} .$$ \end{example}

\section{Proof of the main results}

Recall that $\mu(\Lg)$ is the minimal degree of a faithful representation of $\Lg$ and that we wish to prove the inequality $$\mu(\Lg) \leq d - n + \binom{r + \varepsilon_1}{\varepsilon_1} \cdot \binom{r + \varepsilon_2}{\varepsilon_2}.$$

\begin{proof}(Theorem \ref{TheoremMain}) Let $\map{\iota}{\Lg}{\Lp \ltimes \Lm}$ be the embedding of proposition \ref{PropEmbedding}. 
We may decompose $\Lp \ltimes \Lm$ as in remark \ref{Remark}: $\Lp_0 \oplus (\Lp / \Lp_0 \ltimes \Lm)$. Since the $\mu$-invariant is monotone and sub-additive (cf. \cite{BurdeMoensQuotient}), we obtain the upper bound $$\mu(\Lg) \leq \mu(\Lp \ltimes \Lm) \leq \mu(\Lp_0) + \mu(\Lp / \Lp_0 \ltimes \Lm).$$

Since $\Lp$ acts reductively on itself, it is itself reductive and its reductive ideal $\Lp_0$ satisfies $\mu(\Lp_0) \leq \dim(\Lp_0) \leq \dim(\Lp)$, \cite{BurdeMoensQuotient}. Proposition \ref{PropEmbedding} gives $\dim(\Lp) \leq \dim(\Lg / \Ln)$, so that $\mu(\Lp_0) \leq d - n$. \newline

Proposition \ref{PropEmbedding} guarantees that $\iota(\Lh)$ is an ideal of $\Lp / \Lp_0 \ltimes \Lm$ of codimension $\dim(\Lr / \Lh)$ in $\Lm$. The proposition also gives $\dim(\Lm) = \dim(\Lr)$. Since $\Lp / \Lp_0$ acts faithfully on $\Lm$, we may apply propositions \ref{PropFaithful} and \ref{PropBound} to conclude that $\mu(\Lp / \Lp_0 \ltimes \Lm) \leq \binom{r + \varepsilon_1}{\varepsilon_1} \cdot \binom{r + \varepsilon_2}{\varepsilon_2}$. This finishes the proof. \end{proof}

Note that we obtain an upper bound for $\mu(\Lg)$ that is a polynomial in $d$ of degree $\varepsilon_1 + \varepsilon_2$.


\begin{proof}(Corollary \ref{CorPolynomial}) It suffices to make the following two observations. For natural $\varepsilon_1$ and $\varepsilon_2$ we have the inequality $\lfloor \frac{\varepsilon_1 + \varepsilon_2}{2} \rfloor ! \cdot \lceil \frac{\varepsilon_1 + \varepsilon_2}{2} \rceil! \leq \varepsilon_1 ! \cdot \varepsilon_2!$. Similarly: $(d + \varepsilon_1)! / d! \cdot (d + \varepsilon_2)! / d! \leq (d + \varepsilon_1 + \varepsilon_2)! / d!$.\end{proof}

\section{Application: representations of graded Lie algebras}

A well-known theorem by Jacobson, on weakly closed sets of nilpotent operators, states that a Lie algebra $\Lg$ is nilpotent if it admits a regular derivation, \cite{Jacobson}. (For Lie algebras admitting such a transformation, it is known that $\mu(\Lg) = O(\dim(\Lg))$.) The theorem was later generalized and refined in many different ways, see for example \cite{KostrikinKuznetsov}, \cite{BurdeMoensPeriodic}, and \cite{KhukhroMakarenkoShumyatsky}. 

\begin{theorem}[Khukhro-Makarenko-Shumyatsky -- \cite{KhukhroMakarenkoShumyatsky}] There exist functions $\map{f}{\N \times \N}{\N}$ and $\map{g}{\N}{\N}$ for which the following is true. Consider Lie algebra $\Lg = \oplus_{u \in A} \Lg_u$ graded by a finitely-generated, torsion-free group $(A,+)$. Let $\sigma$ be the cardinality of the support and let $\delta$ be the dimension of the trivial component $\Lg_0$. Then $\Lg$ admits a nilpotent ideal $\Li$ satisfying $\dim(\Lg / \Li) \leq f(\sigma)$ and $c(\Li) \leq g(\sigma,\delta)$. \end{theorem}




\begin{proof}(Corollary \ref{CorGrading}) The nil-defect of $\Lg$ is bounded from above by $\varepsilon(\rad(\Lg),\Li) \leq E(\sigma,\delta) := f(\sigma,\delta) + g(\sigma)$. Note that the functions $f$ and $g$ may grow very quickly as $\sigma$ and $\delta$ increase. \end{proof}


\section*{Aknowledgements}

The author would like to thank his host E. Zel'manov, N. Wallach, E. Khukhro, and G. Glauberman for helpful disucssions.




\end{document}